\documentclass{amsart}
\usepackage{amsmath,amsfonts,amsthm,amssymb,tabularx}
\usepackage{url}
\newcommand{\bbn}{\mathbb{N}}
\newcommand{\bbz}{\mathbb{Z}}
\newcommand{\abs}[1]{\left\lvert #1\right\rvert}

\newcommand{\brac}[1]{\left( #1\right)}
\newtheorem{theorem}{Theorem}
\newtheorem{lemma}{Lemma}
\newtheorem{corollary}{Corollary}
\newtheorem{proposition}{Proposition}
\newtheorem{question}{Question}
\begin{document}

\title{Unit fractions with shifted prime denominators}
\author{Thomas F. Bloom}
\email{bloom@maths.ox.ac.uk}
\address{Mathematical Institute\\Woodstock Road\\Oxford\\OX2 6GG, United Kingdom}

\begin{abstract}
We prove that any positive rational number is the sum of distinct unit fractions with denominators in $\{p-1 : p\textrm{ prime}\}$. The same conclusion holds for the set $\{p-h : p\textrm{ prime}\}$ for any $h\in\mathbb{Z}\backslash\{0\}$, provided a necessary congruence condition is satisfied. We also prove that this is true for any subset of the primes of relative positive density, provided a necessary congruence condition is satisfied.
\end{abstract}

\maketitle

The study of decompositions of rational numbers into sums of distinct unit fractions (often called `Egyptian fractions') is one of the oldest topics in number theory (see \cite{BlEl} for further background and many related problems on Egyptian fractions). It is elementary to prove that such a decomposition is always possible, for instance by using a greedy algorithm. In this paper we explore a natural variant that imposes restrictions on the denominators in these decompositions.
\begin{question}\label{ques}
For which $A\subseteq \bbn$ is it true that every positive rational number can be written as $\sum_{n\in B}\frac{1}{n}$ for some finite $B\subset A$?
\end{question}
A trivial necessary condition is that the set contains multiples of every prime; for example, the set of all odd numbers does not have this property (it cannot represent $\frac{1}{2}$). The condition $\sum_{n\in A}\frac{1}{n}=\infty$ is also clearly necessary, but not sufficient -- indeed, it is easy to see that there is no solution to $1=\sum_{p\in A}\frac{1}{p}$ where $A$ is any finite set of primes, even though $\sum_{p\leq N}\frac{1}{p}\sim \log\log N$. 

For sets $A$ with no such trivial obstructions it is reasonable to speculate that an Egyptian fraction decomposition with denominators restricted to $A$ always exists. An early seminal paper on this topic is by Graham \cite{Gr}, who proved a general result that implies, for example, that such a decomposition always exists when $A$ is the set of all primes \emph{and} squares. Motivated by a conjecture of Sun \cite[Conjecture 4.1]{Su17}, Eppstein \cite{Epp} developed an alternative elementary method, which implies such a decomposition always exists when $A$ is the set of `practical numbers' (those $n$ such that all $m\leq n$ can be written as the sum of distinct divisors of $n$).

A variant of Question~\ref{ques} can be asked even when there are trivial obstructions. For example, Graham \cite{Gr} has shown that every rational number $x$ can be written as the sum of distinct unit fractions with square denominators, subject to the obvious necessary condition that $x\in [0,\frac{\pi^2}{6}-1)\cup [1,\frac{\pi^2}{6})$, and every rational number $x$ with square-free denominator can be written as the sum of distinct unit fractions with square-free denominators.

A natural candidate of number theoretic interest, for which there exist no obvious obstructions to any rational decomposition, and for which the methods of \cite{Gr} and \cite{Epp} are not applicable, is the set of shifted primes $\{p-1:p\textrm{ prime}\}$. That such a restricted Egyptian fraction decomposition always exists was conjectured by Sun \cite[Conjecture 4.1]{Su17} (see also \cite[Conjecture 8.17]{Su21} and \cite{Su15} for some numerical data). In this paper we use the method of \cite{Bl} to prove this conjecture: any positive rational $x>0$ has a solution (indeed, infinitely many) to
\[x=\frac{1}{p_1-1}+\cdots+\frac{1}{p_k-1}\]
where $p_1<\cdots<p_k$ are distinct primes. We also prove a similar result with denominators $p_i-h$ for any (fixed) $h\neq 0$, although for $\abs{h}>1$ there are some trivial congruence obstructions -- for example, since no subset of $\{p+2 : p\textrm{ prime}\}$ has lowest common multiple divisible by $8$ the fraction $\frac{1}{8}$ cannot be represented as the sum of distinct unit fractions of the shape $\frac{1}{p+2}$. 

We deduce this existence result from the following more general result, showing that any shifted set of primes, all divisible by $q$, of `positive upper relative logarithmic density' contains a decomposition of $\frac{1}{q}$. (Recall that $\sum_{p\leq N}\frac{1}{p}\sim \log\log N$, and so it is natural to consider $\sum_{\substack{p\leq N\\ p\in A}}\frac{1}{p}$ divided by $\log\log N$ as a measure of the size of $A$.)
\begin{theorem}\label{th-main}
Let $h\in\bbz\backslash\{0\}$ and $q\geq 1$ be such that $(\abs{h},q)=1$. If $A$ is a set of primes congruent to $h\pmod{q}$ such that
\[\limsup_{N\to\infty} \frac{\sum_{p\in A\cap[1,N]}\frac{1}{p}}{\log\log N}>0\]
then there exists a finite $S\subset A$ such that
\[\frac{1}{q}=\sum_{p\in S}\frac{1}{p-h}.\]
\end{theorem}

A simple application of partial summation produces the following version with (relative) upper logarithmic density replaced by (relative) lower density.

\begin{corollary}\label{cor-one}
Let $h\in\bbz\backslash\{0\}$ and $q\geq 1$ be such that $(\abs{h},q)=1$. If $A$ is a set of primes congruent to $h\pmod{q}$ with positive relative lower density, that is,
\[\liminf_{N\to\infty}\frac{\abs{A\cap [1,N]}}{N/\log N}>0,\]
then there exists a finite $S\subset A$ such that
\[\frac{1}{q}=\sum_{p\in S}\frac{1}{p-h}.\]
\end{corollary}

We remark that (unlike the statement for unrestricted sets of integers, see \cite[Theorem 2]{Bl}) the stronger version of Corollary~\ref{cor-one} with the $\liminf_{N\to\infty}$ replaced by $\limsup_{N\to\infty}$ is false - for example, if $A[N]$ is the set of primes in $[N/2,N]$ then $\sum_{p\in A[N]}\frac{1}{p}\ll \frac{1}{\log N}$, and hence if $A=\cup_k A[N_k]$ where $N_k=2^{k^C}$ for some large absolute constant $C>0$ then $\sum_{n\in A}\frac{1}{n}<1$, and hence certainly we cannot find a finite $S\subset A$ such that $\sum_{n\in S}\frac{1}{n}=1$, and yet
\[\limsup_{N\to\infty}\frac{\abs{A\cap [1,N]}}{N/\log N}\geq 1/2.\]

We now show how Theorem~\ref{th-main} implies the headline result: any positive rational number (subject to the necessary congruence conditions) can be written as the sum of distinct unit fractions with shifted prime denominators.
\begin{corollary}\label{cor-two}
Let $h\in \mathbb{Z}$ and $x=r/q\in \mathbb{Q}_{>0}$ be such that $(\abs{h},q)=1$. There are distinct primes $p_1,\ldots,p_k$ such that
\[x=\frac{1}{p_1-h}+\cdots+\frac{1}{p_k-h}.\]
\end{corollary}
\begin{proof}
By Dirichlet's theorem (see for example \cite[Corollary 4.12]{MV}) if $A$ is the set of primes congruent to $h\pmod{q}$, then
\[\sum_{n\in A\cap[1,N]}\frac{1}{n}\geq (\tfrac{1}{\phi(q)}+o(1))\log\log N.\]
Trivially the same must hold for $A\backslash B$, for any finite set $B$. In particular by $r$ repeated applications of Theorem~\ref{th-main} (first to $A$, then $A\backslash S_1$, and so on) we can find $r$ disjoint finite sets $S_1,\ldots,S_r\subset A$ such that
\[\frac{1}{q}=\sum_{p\in S_i}\frac{1}{p-h}\]
for $1\leq i\leq r$. It follows that
\[x=\frac{r}{q}=\sum_{p\in \bigcup S_i}\frac{1}{p-h}\]
as required.
\end{proof}

We prove Theorem~\ref{th-main} with an application of the author's earlier work \cite{Bl} (which in turn is a stronger form of an argument of Croot \cite{Cr2003}). Loosely speaking, the main result of \cite{Bl} shows that we can solve $1=\sum \frac{1}{n_i}$ with $n_i\in A$ whenever $A$ satisfies
\begin{enumerate}
\item $\sum_{n\in A}\frac{1}{n}\to \infty$, 
\item every $n\in A$ is `friable' (or `smooth'), in that if a prime power $q$ divides $n$ then $q\leq n^{1-\delta(n)}$ for some $0<\delta(n)=o(1)$, 
\item every $n\in A$ has `small divisors', and 
\item every $n\in A$ has $\approx \log\log n$ many distinct prime divisors. 
\end{enumerate}
To prove Theorem~\ref{th-main}, therefore, it suffices to show that the set $\{\frac{p-h}{q} : p \in A\}$ has these properties. Fortunately, there has been a great deal of study of the arithmetic properties of shifted primes, and so using classical techniques from analytic number theory we are able to find a large subset of our original set $A$ satisfying all four properties. 

For experts in analytic number theory we add that in establishing the necessary number theoretic facts about shifted primes we have followed the simplest path, forgoing many of the more elaborate refinements possible. The main observation of this paper is that the inputs required to the method of \cite{Bl} are mild enough to be provable for the shifted primes using (a crude form of) existing technology. 

To minimise technicalities we have proved only a qualitative form of Theorem~\ref{th-main}. In principle a (very weak) quantitative version could be proved with the same methods, along similar lines to \cite[Theorem 3]{Bl}, but this would complicate the presentation significantly.

Finally, the methods and main results of \cite{Bl} have now been formally verified using the Lean proof assistant, in joint work with Bhavik Mehta.\footnote{The formalised proof can be found at \url{https://github.com/b-mehta/unit-fractions}.} This formalisation has not been extended to the present work, but since the proof of Theorem~\ref{th-main} uses the main result of \cite{Bl} as its primary ingredient (combined with classical number theory) it can be viewed as `partially formally verified'. 

In Section~\ref{sec-proof} we prove Theorem~\ref{th-main} assuming certain number theoretic lemmas. In Section~\ref{sec-nt} we prove these lemmas.

\subsection*{Acknowledgements}
The author is funded by a Royal Society University Research Fellowship. We would like to thank Greg Martin for a helpful conversation about friable values of shifted primes and remarks on an earlier version of this paper.
\section{Proof of Theorem~\ref{th-main}}\label{sec-proof}

Our main tool is the following slight variant of \cite[Proposition 1]{Bl} (which is identical to the below except that the exponent of $c$ is replaced by $1/\log\log N$). 
\begin{proposition}\label{th-techmain}
Let $c\in(0,1/4)$ and $N$ be sufficiently large (depending only on $c$). Suppose $A\subset [N^{1-c},N]$ and $1\leq y\leq z\leq (\log N)^{1/500}$ are such that
\begin{enumerate}
\item $\sum_{n\in A}\frac{1}{n}\geq 2/y+(\log N)^{-1/200}$,
\item every $n\in A$ is divisible by some $d_1$ and $d_2$ where $y\leq d_1$ and $4d_1\leq d_2\leq z$,
\item every prime power $q$ dividing some $n\in A$ satisfies $q\leq N^{1-4c}$, and
\item every $n\in A$ satisfies
\[\tfrac{99}{100}\log\log N\leq \omega(n) \leq 2\log\log N.\]
\end{enumerate}
There is some $S\subseteq A$ such that $\sum_{n\in S}\frac{1}{n}=1/d$ for some $d\in [y,z]$.
\end{proposition}
\begin{proof}
The proof is identical to that of \cite[Proposition 1]{Bl}, except that in the final part of the proof we choose $M=N^{1-c}$. Observe that the inputs to that proof, namely \cite[Proposition 2, Proposition 3, and Lemma 7]{Bl}, are valid for any $M\in (N^{3/4},N)$.  It remains to check the `friable' hypothesis, for which we require that if $n\in A$ and $q$ is a prime power with $q\mid n$ then, for some small absolute constant $c>0$,
\[q\leq c\min\brac{\frac{M}{z},\frac{M}{(\log N)^{1/100}},\frac{M^3}{N^{2-4/\log\log N}(\log N)^{2+1/50}}}.\]
For $N$ sufficiently large (depending only on $c$) the right-hand side is $>N^{1-4c}$, and so hypothesis (3) suffices. 
\end{proof}

It is convenient to recast this in a slightly different form.

\begin{proposition}\label{prop-tech}
Let $\delta,\epsilon>0$ and suppose $y$ is sufficiently large depending on $\delta$ and $\epsilon$, and $y\leq w\leq z$. If $N$ is sufficiently large (depending on $\delta,\epsilon,y,w,z$) and $A\subset [2,N]$ is such that for all $n\in A$
\begin{enumerate}
\item if a prime power $q$ divides $n$ then $q\leq n^{1-\epsilon}$, 
\item $\abs{\omega(n)-\log\log n}\leq \log\log n/1000$, 
\item $n$ is divisible by some $d_1\in [y,w)$,
\item $n$ is divisible by some $d_2\in [4w,z)$, and
\item $\sum_{n\in A}\frac{1}{n}\geq \delta\log\log N$,
\end{enumerate}
then there exists $S\subseteq A$ such that $\sum_{n\in S}\frac{1}{n}=1/d$ for some $d\leq z$. 
\end{proposition}
\begin{proof}
For $i\geq 0$ let $N_i=N^{(1-\epsilon/4)^i}$, and let $A_i=A\cap (N_{i+1},N_i]$. Note that $N_i<2$ for $i\geq C\log\log N$, where $C$ is some sufficiently large constant depending only on $\epsilon$. Since $\sum_{n\leq \log\log N}\frac{1}{n}\ll \log\log\log N$ it follows by the pigeonhole principle that there must exist some $i$ such that with $A'=A_i$ and $N'=N_i\gg \log\log N$ we have
\[\sum_{n\in A'}\frac{1}{n}\gg_{\delta,\epsilon} 1\]
and $A'\subset ((N')^{1-\epsilon/4},N']$. It suffices to verify that the assumptions of Proposition~\ref{th-techmain} are satisfied by $A'$, with $c=\epsilon/4$. We have already verified the first assumption (assuming $y$ and $N$ are sufficiently large; note that since $N'\gg \log\log N$ this ensures that $N'$ is also sufficiently large). The second assumption of Proposition~\ref{th-techmain} is ensured by conditions (3) and (4).

For the third assumption, note that by condition (1) if $n\in A'$ is divisible by a prime power $q$ then
\[q\leq n^{1-\epsilon}\leq (N')^{1-\epsilon}\]
as required. Finally the fourth assumption follows from condition (2) and noting that for all $n\in [(N')^{1-\epsilon/4},N']$ we have
\[\log\log n= \log\log N'+O_\epsilon(1),\]
and the $O_\epsilon(1)$ term is $\leq \log\log N'/500$, say, provided we take $N$ sufficiently large.
\end{proof}

To prove Theorem~\ref{th-main} we want to apply Proposition~\ref{prop-tech} to $B=\{ \frac{p-h}{q} : p\in A\}$. To verify the hypotheses we will require the following number-theoretic lemmas. We were unable to find these exact statements in the literature, so have included proofs in the following section, but the proofs are all elementary and cover well-trodden ground.

\begin{lemma}\label{lem-nt1}
For any $\epsilon>0$ and $h\in\bbz\backslash\{0\}$ the relative density of primes $p$ such that $n=p-h$ is divisible by a prime power $q>n^{1-\epsilon}$ is $O_h(\epsilon)$. 
\end{lemma}

\begin{lemma}\label{lem-nt2}
For any $\delta>0$ and $h\in\bbz\backslash\{0\}$ the relative density of primes $p$ such that $n=p-h$ has
\[\abs{\omega(n)-\log\log(n)}\geq \delta\log\log n\]
is $0$. 
\end{lemma}

\begin{lemma}\label{lem-nt3}
For any $h\in\bbz\backslash\{0\}$, if $4\leq y<z$ the relative density of primes $p$ such that $n=p-h$ is not divisible by any primes $q\in [y,z]$ is $O_h(\log y/\log z)$. 
\end{lemma}

We will now show how these lemmas, combined with Proposition~\ref{prop-tech}, imply Theorem~\ref{th-main}.

\begin{proof}[Proof of Theorem~\ref{th-main}]
By assumption there is some $\delta>0$ and infinitely many $N$ such that
\[\sum_{p\in A\cap[1,N]}\frac{1}{p}\geq 4\delta \log\log N.\]
Let $B=\{\frac{p-h}{q} : p \in A\}\subset \bbn$, so that there must exist infinitely many $N$ such that
\[\sum_{n\in B\cap[1,N]}\frac{1}{n}\geq 3\delta \log\log N.\]

Let $\epsilon=c\delta$ where $c>0$ is some small absolute constant to be determined later. Let $y$ be sufficiently large in terms of $\delta$ (so that Proposition~\ref{prop-tech} can apply) and $w\leq z$ be determined shortly, and let $B'\subseteq B$ be the set of those $n\in B$ such that 
\begin{enumerate}
\item if a prime power $r$ divides $n$ then $r\leq n^{1-\epsilon}$, 
\item $\abs{\omega(n)-\log\log n}\leq \log\log n/1000$, 
\item $n$ is divisible by some prime $p_1\in [y,w)$, and
\item $n$ is divisible by some prime $p_2\in [4w,z)$. 
\end{enumerate}
If $X_1$ is the set of $m=p-h$ which are divisible by some prime power $r>m^{1-2\epsilon}$ then by Lemma~\ref{lem-nt1} we have
\[\abs{X_1\cap [1,N]}\ll \epsilon\frac{N}{\log N},\]
and hence, since for all large primes $p$ we have $(p-h)^{1-2\epsilon}\leq (\frac{p-h}{q})^{1-\epsilon}$, the set $B_1$ of those $n\in B$ which fail the first condition satisfies
\[\abs{B_1\cap [1,N]}\ll \epsilon\frac{N}{\log N},\]
whence by partial summation
\[\sum_{n\in B_1\cap[1,N]}\frac{1}{n}\ll \epsilon\log\log N.\]
By a similar argument (recalling that $q$ is some fixed constant, and so $\omega(\frac{p-h}{q})=\omega(p-h)+O(1)$ and $\log\log(\frac{p-h}{q})=\log\log(p-h)+O(1)$),  Lemma~\ref{lem-nt2} implies that the sum of reciprocals from those $n\in B\cap[1,N]$ which fail the second condition is $o(\log\log N)$. Similarly, by Lemma~\ref{lem-nt3} we can choose $w$ and $z$ (depending only $\delta$) such that for all large $N$ the sum of reciprocals from those $n\in B\cap[1,N]$ which fail either condition (3) or (4) is $\leq \delta\log\log N$.  Therefore,  there exist infinitely many $N$ such that (provided $\epsilon$ is a small enough multiple of $\delta$)
\[\sum_{n\in B'\cap [1,N]}\frac{1}{n}\geq 2\delta\log\log N.\]

Fix such an $N$ and let $B''=B'\cap [1,N]$. All of the conditions from Proposition~\ref{prop-tech} are now satisfied for $B''$, and hence there exists some $S_1\subseteq B''$ and $d_1\leq z$ such that $\sum_{n\in S_1}\frac{1}{n}=\frac{1}{d_1}$. 

We now apply Proposition~\ref{prop-tech} again to $B''\backslash S_1$, and continue this process $k=\lceil z\rceil^2$ many times, producing some disjoint $S_1,\ldots,S_k$ and associated $d_1,\ldots,d_k\leq z$ where $\sum_{n\in S_i}\frac{1}{n}=\frac{1}{d_i}$ for $1\leq i\leq k$. Notice that the conditions of Proposition~\ref{prop-tech} remain satisfied for each $B''\backslash \cup_{i\leq j}S_i$ for $j\leq k$, since
\[\sum_{n\in \cup_{i\leq j}S_i}\frac{1}{n}\leq k\ll z^2< \delta\log\log N,\]
assuming $N$ is sufficiently large, since $z$ depends on $\delta$ only.

By the pigeonhole principle there must exist some $d\leq z$ and $i_1,\ldots,i_d$ such that $d_{i_j}=d$ for $1\leq j\leq d$, and hence $S=\cup_{1\leq j\leq d}S_{i_j}$ satisfies
\[\sum_{n\in S}\frac{1}{n}=d\cdot \frac{1}{d}=1\]
as required. 
\end{proof}

\section{Number theoretic ingredients}\label{sec-nt}
It remains to prove Lemmas~\ref{lem-nt1}, \ref{lem-nt2}, and \ref{lem-nt3}, which we will do in turn. 
\subsection{Friability of shifted primes}
There has been a great deal of work on shifted primes with only small prime divisors. Often the focus is on an existence result, finding the smallest possible $\delta>0$ such that there exist infinitely many shifted primes $p-1$ with no prime divisors $>p^\delta$. We refer to \cite{Li} for recent progress on this and references to earlier work. Our focus is a little different: we are content with a very high friability threshold, but we need to show that almost all shifted primes are this friable. For the regime of friability that we are interested even the original elementary methods of Erd\H{o}s \cite{Er35} suffice.
\begin{proof}[Proof of Lemma~\ref{lem-nt1}]
This is only a slight generalisation of \cite[Lemma 4]{Er35}. It suffices to show that, for all $\epsilon>0$ and large $N$, the number of $p\leq N$ such that $p-h$ is divisible by some prime power $q$ with $q>N^{1-\epsilon}$ is
\[\ll_h \epsilon\frac{N}{\log N}.\]
We first note that trivially for any $q$ the number of $p\leq N$ such that $p-h$ is divisible by $q$ is certainly $O_h(N/q)$, and hence the count of those $p-h$ divisible by some non-prime prime power $q>N^{1-\epsilon}$ is
\[\ll_h N\sum_{k\geq 2}\sum_{N^{1-\epsilon}\leq m^k\leq N}\frac{1}{m^k}\ll N^{\epsilon}\ll \epsilon\frac{N}{\log N}\]
for all large $N$. It remains to bound the count of those $p\leq N$ such that $p-h$ is divisible by some prime $q>N^{1-\epsilon}$. Such $p-h$ we can write uniquely (assuming $N$ is large enough depending on $h$) as $p-h=qa$ for some $a\leq 2N^{\epsilon}$ and $q>N^{1-\epsilon}$ prime. A simple application of Selberg's sieve (for example \cite[Theorem 3.12]{HR}) yields that, for any fixed $a\geq 1$ and $h\neq 0$ the number of primes $q\leq x$ such that $aq+h$ is also prime is
\[\ll_h \frac{a}{\phi(a)}\frac{x}{(\log x)^2}.\]
Since $q\leq N/a+O_h(1)$, the number of $p\leq N$ such that $p-h=qa$ is
\[\ll \frac{1}{\phi(a)}\frac{N}{(\log N)^2}.\]
Summing over all $a\leq 2N^\epsilon$ the total count is
\[\ll_h \frac{N}{(\log N)^2}\sum_{a\leq 2N^\epsilon}\frac{1}{\phi(a)}\ll \epsilon\frac{N}{\log N}\]
as required, using the fact that $\sum_{a\leq M}\frac{1}{\phi(a)}\ll \log M$. 
\end{proof}

\subsection{Number of prime divisors of shifted primes}
We need to know that $\omega(n)\sim \log\log n$ for almost all $n\in \{p-h: p\textrm{ prime}\}$. This is in fact the typical behaviour of $\omega(n)$ for a generic integer $n$, and we expect the same behaviour when restricting $n$ to the random-like sequence of shifted primes. Indeed, just like $\omega(n)$ itself, $\omega(p-h)$ satisfies an Erd\H{o}s-Kac theorem, that is, $\omega(p-h)$ behaves like a normal distribution with mean $\log\log(p-h)$ and standard deviation $\sqrt{\log\log(p-h)}$. This was established by Halberstam \cite{Ha}, although a simple variance bound suffices for our application here.

\begin{proof}[Proof of Lemma~\ref{lem-nt2}]
It suffices to show that, for all $\delta>0$ and large $N$, if $A$ is the set of $p\leq N$ such that $\abs{\omega(p-h)-\log\log (p-h)}>\delta \log\log (p-h)$, then 
\[\abs{A}\ll \frac{N}{(\log N)(\log\log N)}.\]
Let $A_1=A\cap [1,N^{1/2}]$ and $A_2=A\backslash A_1$. We can trivially bound $\abs{A_1}\ll N^{1/2}$, and for $p\in A_2$ we have $\log\log(p-h)=\log\log N+O(1)$, whence for large enough $N$ if $p\in A_2$ we have
\[\abs{\omega(p-h)-\log\log N}>\tfrac{\delta}{2}\log\log N.\]
By \cite[Theorem 3]{Ha}, however, we have  
\[\sum_{p\leq N}\abs{\omega(p-h)-\log\log N}^2\ll \pi(N)\log\log N,\]
and hence
\[\abs{A_2}(\log\log N)^2\ll_\delta \pi(N)\log\log N,\]
and the result now follows from Chebyshev's estimate $\pi(N)\ll N/\log N$. 
\end{proof}

\subsection{Shifted primes with small divisors}
For Lemma~\ref{lem-nt3} we need to show that there are few shifted primes remaining after we remove all multiples of primes $p\in [y,z]$, which is a classic upper bound sieve problem. Since the information we require is very weak even the simplest sieve suffices: the following is proved as \cite[Theorem 1.1]{HR}.
\begin{lemma}[Sieve of Eratosthenes-Legendre]\label{lem-sieve}
Let $A$ be a finite set of integers and $\mathcal{P}$ a finite set of primes. Let $z\geq 2$ and $P(z)=\prod_{\substack{p\in \mathcal{P}\\ p<z}}p$. Suppose that $f(d)$ is a multiplicative function and $X>1$ is such that for all $d\mid P(z)$ we have
\[\abs{A_d}=f(d)X+R_d.\]
Then
\[\#\{n\in A : (n,P(z))=1\}\ll X\prod_{\substack{p\in P\\p<z}}\brac{1-f(p)}+\sum_{d\mid P(z)}\abs{R_d}.\] 
\end{lemma}
For the required sieve input we will use the following classic result on the distribution of primes within arithmetic progressions (which is proved, for example, as \cite[Corollary 11.21]{MV}). Recall that $\pi(N;d,h)$ is the number of primes $p\leq N$ such that $p\equiv h\pmod{d}$. 
\begin{theorem}[Siegel-Walfisz]\label{th-sw}
There is a constant $c>0$ such that for all $h\in \bbz$ and $1\leq d\leq \log x$ with $(\abs{h},d)=1$ we have 
\[\pi(N;d,h)=\frac{\mathrm{li}(N)}{\phi(d)}+O(N\exp(-c\sqrt{\log N})).\]
\end{theorem}	

\begin{proof}[Proof of Lemma~\ref{lem-nt3}]
Fix $4\leq y\leq z$ and let $P=\prod_{\substack{y\leq q\leq z\\ q\nmid h}}q$ (where $q$ is restricted to primes). It suffices to show that, for all large $N$, 
\[\#\{ p-h\leq N : (p-h,P)=1\}\ll_h \frac{\log y}{\log z}\mathrm{li}(N).\]
We will apply Lemma~\ref{lem-sieve} with $A=\{p-h : p\leq N\}$, 
\[\mathcal{P}=\{ p\in [y,z] : p\nmid h\},\]
$f(d)=1/\phi(d)$, and $X=\mathrm{li}(N)$, noting that by Theorem~\ref{th-sw} whenever $(d,h)=1$ and $d\leq \log N$
\[\abs{A_d}=\pi(N;d,h)=\frac{\mathrm{li}(N)}{\phi(q)}+O(x\exp(-c\sqrt{\log x})).\]
It follows that
\[\#\{ p-h\leq N : (p-h,P)=1\}\ll \mathrm{li}(N) \prod_{\substack{y\leq q\leq z\\ q\nmid h}}\brac{1-\frac{1}{q-1}}+2^zN\exp(-c\sqrt{\log N}).\]
The conclusion now follows provided we choose $N$ large enough so that $z\ll\sqrt{\log \log N}$, say, and using Mertens' estimate that
\[\prod_{p\leq w}(1-1/p)\asymp\frac{1}{\log w}.\]
\end{proof}

\end{document}